\documentclass[11pt]{amsart}

\usepackage[utf8]{inputenc}
\usepackage[T1]{fontenc}
\usepackage{amsmath,amssymb,amsthm}
\usepackage{mathtools}
\usepackage{geometry}
\usepackage[colorlinks=true,linkcolor=blue,citecolor=blue,urlcolor=blue]{hyperref}
\usepackage{enumitem}
\usepackage{cleveref}

\newtheorem{theorem}{Theorem}[section]
\newtheorem{proposition}[theorem]{Proposition}
\newtheorem{lemma}[theorem]{Lemma}

\newtheorem{example}[theorem]{Example}
\newtheorem{remark}[theorem]{Remark}

\newcommand{\R}{\mathbb{R}}
\newcommand{\C}{\mathbb{C}}
\newcommand{\Z}{\mathbb{Z}}
\newcommand{\nn}{\mathfrak{n}}
\newcommand{\frakg}{\mathfrak{g}}
\newcommand{\hh}{\mathfrak{h}}
\newcommand{\mm}{\mathfrak{m}}
\newcommand{\zz}{\mathfrak{z}}
\newcommand{\fraka}{\mathfrak{a}}
\newcommand{\so}{\mathfrak{so}}

\newcommand{\der}{\operatorname{Der}}
\newcommand{\ann}{\operatorname{Ann}}
\newcommand{\ad}{\operatorname{ad}}
\newcommand{\Ad}{\operatorname{Ad}}
\newcommand{\End}{\operatorname{End}}
\newcommand{\Aut}{\operatorname{Aut}}

\newcommand{\Span}{\operatorname{span}}
\newcommand{\ip}[2]{\langle #1, #2 \rangle}
\newcommand{\proj}[1]{[\,#1\,]_{\mm}}

\begin{document}

\title{Cyclic Riemannian nilmanifolds are naturally reductive}

\author{Hui Wang}
\address{College of Science, Nanjing University of Posts and Telecommunications, Nanjing 210003, People’s Republic of China}
\email{wanghui0801@njupt.edu.cn}

\author{Zaili Yan$^{*}$}
\address{School of Mathematics and Statistics, Ningbo University,
Ningbo, Zhejiang Province, 315211, People's Republic of China}
\email{yanzaili@nbu.edu.cn}
\thanks{Z.~Yan$^{*}$ is the corresponding author and is supported by the Zhejiang
Provincial Natural Science Foundation of China under Grant No.~LMS25A010010.}

\author{Shaoxiang Zhang}
\address{College of Mathematics and Systems Science,
Shandong University of Science and Technology, Qingdao 266590,
People's Republic of China}
\email{zhangshaoxiang93@163.com}
\thanks{S.~Zhang is partially supported by the Natural Science Foundation of Shandong Province (No. ZR2026MS0024) and by the Science and Technology Support
Plan for Youth Innovation of Colleges and Universities of Shandong Province
of China (No.~2023KJ090).}

\subjclass[2010]{Primary 53C30; Secondary 53C25, 53C24}

\keywords{Riemannian nilmanifold; cyclic metric; naturally reductive}

\begin{abstract}
We study the structure of cyclic Riemannian nilmanifolds, that is, connected
nilpotent Lie groups $N$ endowed with a left-invariant cyclic metric
$\ip{\cdot}{\cdot}$. We prove that a Riemannian nilmanifold
$(N,\ip{\cdot}{\cdot})$ is cyclic if and only if its Lie algebra $\nn$ is at
most two-step nilpotent and the associated family of skew-symmetric
endomorphisms $J_{\zz}=\{J_Z : Z\in \zz\}\subset \so(\fraka)$ is Abelian. As a
direct consequence, every cyclic Riemannian nilmanifold is naturally
reductive. We also determine the full isometry group of a connected and simply
connected cyclic Riemannian nilmanifold without Euclidean factor.
\end{abstract}

\maketitle

\setcounter{tocdepth}{1}
\tableofcontents

\section{Introduction and main results}\label{sec:intro}

The study of homogeneous Riemannian manifolds is a classical and central topic
in Riemannian geometry. Let $(M,g)$ be a connected homogeneous Riemannian
manifold and let $M=G/H$ be a homogeneous presentation, where $G$ is a
connected Lie group acting transitively and effectively on $M$ by isometries
and $H$ is the isotropy subgroup at a point $p\in M$. Let $\frakg$ and $\hh$
denote the Lie algebras of $G$ and $H$, respectively. A reductive
decomposition of $\frakg$ is a vector space direct sum $\frakg=\hh\oplus \mm$
satisfying $\Ad(H)\mm\subset \mm$.
Identifies $\mm$ as the tangent space $T_{p}M$, the metric $g$ induces an $\Ad(H)$-invariant
inner product $\ip{\cdot}{\cdot}$ on $\mm$; conversely, such an inner product on
$\mm$ determines a $G$-invariant Riemannian metric on $M$. For
$X,Y\in\frakg$, we denote by $\proj{X,Y}$ the projection of $[X,Y]$ onto $\mm$.

Following Gadea, Gonz\'alez-D\'avila and Oubi\~na
\cite{Gadea2015,Gadea2016}, a homogeneous Riemannian manifold $(M,g)$ is said to
be cyclic if it admits a homogeneous presentation $M=G/H$ and a reductive
decomposition $\frakg=\hh\oplus \mm$ such that
\begin{equation}\label{eq:cyclic}
  \ip{\proj{X,Y}}{Z}+\ip{\proj{Y,Z}}{X}+\ip{\proj{Z,X}}{Y}=0,
  \qquad \forall  X,Y,Z\in \mm .
\end{equation}
In other words, the cyclic sum of the trilinear form $(X,Y,Z)\mapsto
\ip{\proj{X,Y}}{Z}$ vanishes identically. When the isotropy is trivial, that is,
when $M=G$ is a Lie group carrying a left-invariant metric and $\hh=\{0\}$,
$\mm=\frakg$, condition \eqref{eq:cyclic} reduces to the purely algebraic condition
\begin{equation}\label{eq:cycliclie}
  \ip{[X,Y]}{Z}+\ip{[Y,Z]}{X}+\ip{[Z,X]}{Y}=0,
  \qquad \forall X,Y,Z\in \frakg,
\end{equation}
and $(G,\ip{\cdot}{\cdot})$ is called a cyclic metric Lie group \cite{Gadea2015}.
It was proved in \cite{Gadea2015} that a cyclic metric nilpotent Lie group must be Abelian.

The notion of cyclicity has its roots in the theory of homogeneous structures, initiated by Ambrose and Singer \cite{AmbroseSinger} and developed
by Tricerri and Vanhecke \cite{TricerriVanhecke}. Recall that a homogeneous
structure on a Riemannian manifold $(M,g)$ is a $(1,2)$-tensor field $S$
satisfying
\[
  \widetilde\nabla g=0,\qquad \widetilde\nabla R=0,\qquad \widetilde\nabla S=0,
\]
where $\widetilde\nabla=\nabla-S$, $\nabla$ is the Levi-Civita connection of
$(M,g)$ and $R$ is its curvature tensor. Ambrose and Singer \cite{AmbroseSinger}
proved that a connected, simply connected and complete Riemannian manifold is
homogeneous if and only if it admits a homogeneous structure. Tricerri and
Vanhecke \cite{TricerriVanhecke} decomposed the space of all possible tensors
$S$ satisfying $\widetilde\nabla g=0$ into three irreducible
$O(\dim M)$-modules
\[
  \mathcal T_1\oplus \mathcal T_2\oplus \mathcal T_3 .
\]
In this classification, a Riemannian manifold admits a homogeneous structure of
type $\mathcal T_3$ if and only if it is naturally reductive \cite{TricerriVanhecke}. Recall that a homogeneous
Riemannian manifold $(M=G/H,g)$ is naturally reductive if there exists a reductive
decomposition $\frakg=\hh\oplus \mm$ satisfying
\begin{equation}\label{eq:natred}
  \ip{\proj{X,Y}}{Z}+\ip{Y}{\proj{X,Z}}=0,
  \qquad \forall  X,Y,Z\in\mm .
\end{equation}

In \cite{Gadea2016}, Gadea, Gonz\'alez-D\'avila and Oubi\~na observed that a
homogeneous structure belongs to the class $\mathcal T_1\oplus\mathcal T_2$ if
and only if $(M,g)$ is cyclic; equivalently, the homogeneous structure
$S=\nabla-\nabla^{c}$ defined by a reductive decomposition $\frakg=\hh\oplus\mm$
is of type $\mathcal T_1\oplus\mathcal T_2$, where $\nabla^{c}$ denotes the
canonical connection associated with the decomposition. They obtained several
characterizations and properties of cyclic homogeneous Riemannian manifolds,
classified the simply connected cyclic homogeneous Riemannian manifolds of
dimension at most four, and exhibited a wide list of examples of noncompact
irreducible Riemannian $3$-symmetric spaces admitting cyclic metrics. We refer
to \cite{Bieszk,FalcitelliPastore,Gadea2015,Gadea2016,KowalskiTricerri,
PastoreVerroca,TricerriVanhecke} for further information and background.

A particularly rich and well-understood family of homogeneous Riemannian
manifolds is given by the Riemannian nilmanifolds, that is, connected
nilpotent Lie groups $N$ endowed with a left-invariant Riemannian metric. The
geometry of such manifolds has been studied intensively; we single out the
foundational work of Wilson \cite{Wilson} on their isometry groups and of Gordon
\cite{Gordon} on their natural reductivity. The aim of the present paper is to
investigate the structure of cyclic Riemannian nilmanifolds and, in
particular, to relate cyclicity to the algebraic invariants introduced by
Gordon.

Let $\nn$ denote the Lie algebra of $N$, identified with the Lie algebra of
left-invariant vector fields on $N$. We shall show in  \cref{prop:twostep} that
the cyclic condition forces $\nn$ to be at most two-step nilpotent. In the
two-step case one introduces the following standard objects. Denote by $\zz$
the center of $\nn$ and let $\fraka=\zz^{\perp}$ be the orthogonal complement of
$\zz$ in $\nn$ with respect to $\ip{\cdot}{\cdot}$. Since $\nn$ is two-step
nilpotent one has $[\nn,\nn]\subset \zz$, and one defines a linear map
\[
  J:\zz\longrightarrow \so(\fraka)=\so(\fraka,\ip{\cdot}{\cdot}|_{\fraka})
\]
by
\[
  \ip{J_Z X}{Y}=\ip{[X,Y]}{Z}, \qquad X,Y\in\fraka,\ Z\in\zz .
\]
The subspace
\[
  J_{\zz}:=\{J_Z : Z\in\zz\}\subset \so(\fraka)
\]
is the central algebraic object of the paper. Our main result is the following.

\begin{theorem}\label{thm:main}
A Riemannian nilmanifold $(N,\ip{\cdot}{\cdot})$ is cyclic if and only if its
Lie algebra $\nn$ is at most two-step nilpotent and $J_{\zz}$ is an Abelian
subalgebra of $\so(\fraka)$. Consequently, every cyclic Riemannian nilmanifold is
naturally reductive.
\end{theorem}

The paper is organized as follows. In \cref{sec:prelim} we collect the
necessary preliminaries on Riemannian nilmanifolds, including Wilson's
description of their isometry groups and Gordon's characterization of natural
reductivity. \Cref{sec:proof} is devoted to the proof of \cref{thm:main}; the
two implications are treated separately, the first one (\cref{prop:twostep})
being a filtration argument and the second one an explicit computation in a
suitable reductive decomposition. \Cref{sec:isom} contains the description of
the isometry group. In \cref{sec:examples} we discuss several examples,
including the Heisenberg group and the family of nilpotent Lie algebras with
diagonal $J$-map, and we comment on the low-dimensional classification and on
the contrast with Heisenberg-type groups.

\section{Preliminaries: Riemannian nilmanifolds and the $J$-map}
\label{sec:prelim}

In this section we recall, in a form suited to our purposes, some standard facts
about Riemannian nilmanifolds. Our main references are \cite{Gordon,Wilson};
see also \cite{Eberlein} for a comprehensive account of two-step nilpotent
groups.

\subsection{Data triples and Wilson's theorem}

A Riemannian nilmanifold is a connected nilpotent Lie group $N$ endowed
with a left-invariant Riemannian metric $\ip{\cdot}{\cdot}$. Such a manifold can
be specified by a data triple $(\nn,\ip{\cdot}{\cdot},L)$, where $\nn$ is
the Lie algebra of $N$, identified with the space of left-invariant vector
fields on $N$, and $L$ is a lattice, that is, a discrete vector subgroup
of the center of $\nn$. Indeed, if $\widetilde N$ denotes the simply connected
Lie group with Lie algebra $\nn$ and $\exp:\nn\to \widetilde N$ is the group
exponential map, then $\exp(L)$ is a discrete central subgroup of $\widetilde
N$; the quotient $N=\widetilde N/\exp(L)$ is a nilpotent Lie group and
$\ip{\cdot}{\cdot}$ defines a left-invariant metric on it. Conversely, every
Riemannian nilmanifold arises in this way, since a connected nilpotent Lie group
with a left-invariant metric has a lattice as above after passing to the
appropriate covering, and the compactness or noncompactness of $N$ is reflected
in the lattice $L$.

The following theorem of Wilson \cite{Wilson} describes the full isometry group
and is the starting point of the theory.

\begin{lemma}[Wilson \cite{Wilson}]\label{lem:wilson}
Let $(N,\ip{\cdot}{\cdot})$ be a Riemannian nilmanifold with data triple
$(\nn,\ip{\cdot}{\cdot},L)$, and let $G$ be its full isometry group.
\begin{enumerate}[label=(\roman*)]
\item $G$ contains a unique simply transitive nilpotent subgroup $N$; moreover,
$N$ coincides with the nilradical of $G$.
\item The Lie algebra $\frakg$ of $G$ is the vector space direct sum
$\frakg=\hh\oplus\nn$, where the isotropy algebra $\hh$ is given by
\[
  \hh=\so(\nn,\ip{\cdot}{\cdot})\cap \der(\nn)\cap \ann(L).
\]
Here $\so(\nn,\ip{\cdot}{\cdot})$ is the Lie algebra of skew-symmetric
endomorphisms of $\nn$, $\der(\nn)$ is the Lie algebra of derivations of $\nn$,
and $\ann(L)=\{\varphi\in\End(\nn): \varphi|_{L}=0\}$ is the annihilator of $L$.
\end{enumerate}
\end{lemma}

Thus, at the infinitesimal level, the full isometry algebra is always a
semidirect sum $\frakg=\hh\oplus\nn$, where $\hh$ acts on $\nn$ by derivations which
are skew-symmetric with respect to the metric and which annihilate the lattice.
This provides a canonical reductive decomposition for any Riemannian
nilmanifold, and the question of natural reductivity becomes the question of
whether this canonical decomposition (or a refinement of it) satisfies
\eqref{eq:natred}.

\subsection{Two-step nilpotent Lie algebras and the $J$-map}

From now on we assume that $\nn$ is two-step nilpotent, that is,
$[\nn,[\nn,\nn]]=\{0\}$, equivalently $[\nn,\nn]\subset \zz$, where $\zz$ is the
center of $\nn$. Let
\[
  \fraka=\zz^{\perp}
\]
be the orthogonal complement of $\zz$ in $\nn$. Since $\zz$ is the center, we
have $[\fraka,\zz]=\{0\}$ and $[\fraka,\fraka]\subset \zz$. Define the linear map
\[
  J:\zz\longrightarrow \so(\fraka),\qquad Z\longmapsto J_Z,
\]
by
\begin{equation}\label{eq:Jmap}
  \ip{J_Z X}{Y}=\ip{[X,Y]}{Z}, \qquad X,Y\in\fraka,\ Z\in\zz .
\end{equation}
The fact that $J_Z\in\so(\fraka)$ is immediate from the skew-symmetry of the Lie
bracket: $\ip{J_Z X}{Y}=\ip{[X,Y]}{Z}=-\ip{[Y,X]}{Z}=-\ip{J_Z Y}{X}$. Thus $J$
completely encodes the bracket of $\nn$: for $X,Y\in\fraka$ and $Z\in\zz$,
\eqref{eq:Jmap} determines $[X,Y]$ through the nondegenerate pairing with
$\zz$.

Set
\[
  \zz_0:=[\nn,\nn]\subset \zz .
\]
Then $\nn=\fraka\oplus\zz$ and $\zz=\zz_0\oplus \ker J$, where $\ker J=\{Z\in\zz:
J_Z=0\}$. Indeed, $Z\in\ker J$ means $J_Z=0$, i.e.\ $\ip{[X,Y]}{Z}=0$ for all
$X,Y\in\fraka$, which is equivalent to $Z\perp [\nn,\nn]=\zz_0$; hence
$\ker J=\zz_0^{\perp}$ and $\zz=\zz_0\oplus\ker J$ orthogonally. Moreover, $J$
is injective on $\zz_0$. Setting $\nn_0=\fraka\oplus\zz_0$, one has $\nn=\nn_0\oplus
\ker J$ as an orthogonal direct sum of ideals, $\ker J$ being central. At the
group level this yields (see \cite[Proposition 4.6]{Gordon})
\[
  \widetilde N\cong \exp(\nn_0)\times \R^{\dim\ker J},
\]
and the direct product is also a product of Riemannian manifolds. Consequently, for many
questions one may reduce to the case in which $J$ is injective, i.e.\
$\zz=\zz_0=[\nn,\nn]$, which means precisely that $\nn$ has no Euclidean
(de Rham) factor.

We now recall two structural results of Gordon \cite{Gordon} that will be used
repeatedly. Throughout, for a two-step nilpotent Lie algebra $\nn$, we keep the notation
$\zz$, $\fraka$, $J$ introduced above.

\begin{lemma}[{Gordon \cite[Lemma 4.7]{Gordon}}]\label{lem:gordon}
Assume that $\nn$ is two-step nilpotent and that $J$ is injective. Then the
isotropy algebra $\hh$ of the full isometry group acts as follows.
\begin{enumerate}[label=(\roman*)]
\item $\hh$ leaves each of the subspaces $\fraka$ and $\zz$ invariant.
\item For every $\varphi\in\hh$,
\[
  \varphi|_{\zz}=J^{-1}\circ \ad_{\so(\fraka)}(\varphi|_{\fraka})\circ J .
\]
In particular, the restriction map $\varphi\mapsto \varphi|_{\fraka}$ is an
isomorphism of $\hh$ onto a subalgebra of $\so(\fraka)$.
\item Let $\varphi\in\so(\fraka)$. Then $\varphi$ extends to an element of $\hh$ if
and only if $[\varphi,J_{\zz}]\subset J_{\zz}$ and
\[
  J^{-1}\circ \ad_{\so(\fraka)}(\varphi)\circ J \in \so(\zz)\cap \ann(L).
\]
\end{enumerate}
\end{lemma}

The content of \cref{lem:gordon} is that an isotropy transformation is
completely determined by its restriction to $\fraka$, and that the admissible
restrictions are precisely those which normalize $J_{\zz}$ and whose induced
action on $\zz$ is skew-symmetric and annihilates the lattice.

In the case where $J$ is not injective, the statements above hold with the
obvious modifications, replacing $\zz$ by $\zz_0$ and leaving $\ker J$ invariant
and pointwise fixed. We shall use this version in the proof of
\cref{thm:main}, where the relevant decomposition is $\nn=\fraka\oplus\zz_0\oplus
\ker J$.

\begin{theorem}[{Gordon \cite[Theorem 4.8]{Gordon}}]\label{thm:gordon}
A Riemannian nilmanifold $(N,\ip{\cdot}{\cdot})$ with data triple
$(\nn,\ip{\cdot}{\cdot},L)$ is naturally reductive if and only if $\nn$ is at
most two-step nilpotent and both of the following conditions hold:
\begin{enumerate}[label=(\roman*)]
\item $J_{\zz}$ is a subalgebra of $\so(\fraka)$;
\item $\tau_Z\in\so(\zz_0)\cap\ann(L)$ for every $Z\in\zz_0$, where the linear
map $\tau_Z:\zz_0\to\zz_0$ is defined by
\[
  [J_Z,J_{Z'}]=J_{\tau_Z Z'}, \qquad Z,Z'\in\zz_0 .
\]
\end{enumerate}
\end{theorem}

\begin{remark}\label{rem:tau}
The map $\tau_Z$ in \cref{thm:gordon} is well defined because $J$ is injective
on $\zz_0$. Note that condition (ii) in \cref{thm:gordon}  is equivalent to the
requirement that $(J_Z,\tau_Z,0|_{\ker J})\in\hh$ for every $Z\in\zz_0$; this is
exactly the content of \cref{lem:gordon}(iii). It follows directly from
$[J_Z,J_{Z'}]=-[J_{Z'},J_Z]$ and the injectivity of $J$ on $\zz_0$ that
\begin{equation}\label{eq:tauantisym}
  \tau_Z Z'=-\tau_{Z'}Z, \qquad Z,Z'\in\zz_0 .
\end{equation}
\end{remark}


\section{Proof of the main theorem}\label{sec:proof}

This section is devoted to the proof of \cref{thm:main}. We proceed in three
steps. First, in \cref{prop:twostep}, we prove that a cyclic Riemannian
nilmanifold is necessarily at most two-step nilpotent. Second, we prove that if
$\nn$ is two-step nilpotent and $J_{\zz}$ is Abelian, then the Riemannian manifold is
cyclic (\cref{lem:suff}). Finally, we prove the converse implication
(\cref{lem:necc}). \Cref{thm:main} then follows, and the natural reductivity
statement is obtained from \cref{thm:gordon}.

\subsection{Cyclicity forces two-step nilpotency}

\begin{proposition}\label{prop:twostep}
Every cyclic Riemannian nilmanifold is at most two-step nilpotent.
\end{proposition}

\begin{proof}
Let $(\nn,\ip{\cdot}{\cdot},L)$ be a data triple for the Riemannian nilmanifold
$(N,\ip{\cdot}{\cdot})$. Suppose that $\nn$ is $s$-step nilpotent and consider
the lower central series
\[
  \nn=\nn^{(0)}\supset \nn^{(1)}\supset \nn^{(2)}\supset \cdots \supset
  \nn^{(s)}=\{0\}, \qquad \nn^{(i+1)}=[\nn,\nn^{(i)}].
\]
Let $(N,\ip{\cdot}{\cdot})$ be cyclic with respect to a subgroup $G_0$ of the
full isometry group $G$ and a reductive decomposition $\frakg_0=\hh_0\oplus\mm$,
where $\hh_0=\frakg_0\cap\hh$. As in \cite[Theorem 4.3]{Gordon}, there exists a
linear map
\[
  \rho:\nn\longrightarrow \hh_0
\]
such that
\[
  \mm=\{X+\rho(X): X\in\nn\}.
\]
Define $\lambda:\nn\to\mm$ by $\lambda(X)=X+\rho(X)$. Since $\mm$ is a
complement to $\hh_0$ and $\frakg_0=\hh_0\oplus\nn$ at the vector-space level, $\lambda$
is a linear isomorphism; moreover, relative to the induced inner product
$\ip{\cdot}{\cdot}_{\mm}$ on $\mm$, $\lambda$ is an isometry, that is,
\begin{equation}\label{eq:isom}
  \ip{\lambda(X)}{\lambda(Y)}_{\mm}=\ip{X}{Y}, \qquad X,Y\in\nn .
\end{equation}

For each $i\in\{0,1,\dots,s-1\}$, let $\nn_{(i)}$ denote the orthogonal
complement of $\nn^{(i+1)}$ in $\nn^{(i)}$, so that
\[
  \nn^{(i)}=\nn_{(i)}\oplus \nn^{(i+1)}
\]
and
\[
  \nn=\bigoplus_{i=0}^{s-1}\nn_{(i)}
\]
is an orthogonal decomposition. Correspondingly, set
$\mm_{(i)}=\lambda(\nn_{(i)})$, so that $\mm=\bigoplus_{i=0}^{s-1}\mm_{(i)}$.

We shall need three elementary facts.

\emph{Fact 1.} For every $i$, one has $[\hh_0,\nn_{(i)}]\subset\nn_{(i)}$.
Indeed, $\hh_0\subset\hh\subset\der(\nn)$, and every derivation preserves the
lower central series: $\varphi(\nn^{(i+1)})=\varphi([\nn,\nn^{(i)}])\subset
[\varphi(\nn),\nn^{(i)}]+[\nn,\varphi(\nn^{(i)})]\subset[\nn,\nn^{(i)}]=\nn^{(i+1)}$
by induction on $i$. Since $\hh_0\subset\so(\nn)$ consists of skew-symmetric
endomorphisms, it also preserves orthogonal complements, whence
$[\hh_0,\nn_{(i)}]\subset\nn_{(i)}$.

\emph{Fact 2.} For $X\in\nn_{(i)}$ and $Y\in\nn_{(j)}$, one has
$[X,Y]\in\nn^{(i+j+1)}$. This is the standard inclusion
$[\nn^{(i)},\nn^{(j)}]\subset\nn^{(i+j+1)}$ for the lower central series.

\emph{Fact 3.} For $X\in\nn_{(i)}$ and $Y\in\nn_{(j)}$, the projection onto
$\mm$ of $[\lambda(X),\lambda(Y)]$ satisfies
\begin{equation}\label{eq:proj}
  \proj{\lambda(X),\lambda(Y)}=
  \lambda\bigl([X,Y]+[\rho(X),Y]-[\rho(Y),X]\bigr)
\end{equation}
and therefore
\begin{equation}\label{eq:3.2}
  \proj{\lambda(X),\lambda(Y)}\ \in\
  \Bigl(\bigoplus_{k\ge i+j+1}\mm_{(k)}\Bigr)\oplus \mm_{(j)}\oplus \mm_{(i)} .
\end{equation}
Indeed, since $\rho(X),\rho(Y)\in\hh_0$ and $\hh_0$ is a subalgebra,
\begin{align*}
  [\lambda(X),\lambda(Y)]
  &=[X+\rho(X),\,Y+\rho(Y)]\\
  &=[X,Y]+[\rho(X),Y]-[\rho(Y),X]+[\rho(X),\rho(Y)],
\end{align*}
where we used $[X,\rho(Y)]=-[\rho(Y),X]$. The last term lies in $\hh_0$, hence
projects to $0$ in $\mm$, while $[X,Y]\in\nn^{(i+j+1)}$ by Fact 2,
$[\rho(X),Y]\in\nn_{(j)}$ and $[\rho(Y),X]\in\nn_{(i)}$ by Fact 1. Applying the
isometry $\lambda$ yields \eqref{eq:3.2}.

We now exploit the cyclic condition \eqref{eq:cyclic}. Fix $X\in\nn_{(i)}$ and
let $B\in\bigoplus_{k<i}\mm_{(k)}$ and $Y\in\nn_{(l)}$ with $l\ge i+1$. By
\eqref{eq:cyclic} applied to the triple $(\lambda(X),B,\lambda(Y))$,
\[
  \ip{\proj{\lambda(X),B}}{\lambda(Y)}_{\mm}
  =-\ip{\proj{B,\lambda(Y)}}{\lambda(X)}_{\mm}
   -\ip{\proj{\lambda(Y),\lambda(X)}}{B}_{\mm} .
\]
We claim that both terms on the right-hand side vanish. Writing $B=\lambda(U)$
with $U\in\nn_{(k)}$ and $k<i$, \eqref{eq:3.2} gives
\[
  \proj{B,\lambda(Y)}=\proj{\lambda(U),\lambda(Y)}\in
  \Bigl(\bigoplus_{r\ge k+l+1}\mm_{(r)}\Bigr)\oplus\mm_{(l)}\oplus\mm_{(k)} .
\]
Since $k<i$ and $l\ge i+1$, all indices occurring on the right are different
from $i$; hence this element is orthogonal to $\lambda(X)\in\mm_{(i)}$, and the
first term vanishes. Likewise, by \eqref{eq:3.2} with the roles of $X$ and $Y$
interchanged,
\[
  \proj{\lambda(Y),\lambda(X)}\in
  \Bigl(\bigoplus_{r\ge l+i+1}\mm_{(r)}\Bigr)\oplus\mm_{(i)}\oplus\mm_{(l)} ,
\]
all of whose components have index at least $i$; in particular it is orthogonal
to $B\in\bigoplus_{k<i}\mm_{(k)}$, so the second term vanishes. Therefore
\[
  \ip{\proj{\lambda(X),B}}{\lambda(Y)}_{\mm}=0 .
\]
As $\lambda(Y)$ ranges over $\bigoplus_{l\ge i+1}\mm_{(l)}$ and the inner product
is nondegenerate, we deduce that
\begin{equation}\label{eq:lower}
  \Bigl[\lambda(X),\,\bigoplus_{k<i}\mm_{(k)}\Bigr]_{\mm}\subset
  \bigoplus_{k\le i}\mm_{(k)} .
\end{equation}

Now take $X\in\nn_{(i)}$ and $Y\in\nn_{(j)}$ with $j<i$. On the one hand,
\eqref{eq:lower} gives
$\proj{\lambda(X),\lambda(Y)}\in\bigoplus_{k\le i}\mm_{(k)}$. On the other
hand, by \eqref{eq:3.2},
\[
  \proj{\lambda(X),\lambda(Y)}=
  \lambda([X,Y])+\lambda([\rho(X),Y])-\lambda([\rho(Y),X])
\]
with $\lambda([X,Y])\in\bigoplus_{k\ge i+j+1}\mm_{(k)}\subset
\bigoplus_{k>i}\mm_{(k)}$, $\lambda([\rho(X),Y])\in\mm_{(j)}$ and
$\lambda([\rho(Y),X])\in\mm_{(i)}$. Since the decomposition $\mm=\bigoplus
\mm_{(k)}$ is direct, the component $\lambda([X,Y])$, which lies in
$\bigoplus_{k>i}\mm_{(k)}$, must vanish. As $\lambda$ is injective, we obtain
$[X,Y]=0$. Interchanging the roles of $i$ and $j$, we conclude that
\begin{equation}\label{eq:crosszero}
  [\nn_{(i)},\nn_{(j)}]=\{0\}, \qquad i\neq j .
\end{equation}

Consequently,
\[
  \nn^{(1)}=[\nn,\nn]=\sum_{i,j}[\nn_{(i)},\nn_{(j)}]
  =\sum_{i}[\nn_{(i)},\nn_{(i)}] .
\]
Finally, using the Jacobi identity and \eqref{eq:crosszero},
\[
  \nn^{(2)}=[\nn,\nn^{(1)}]
  =\sum_{i}\bigl[\nn,[\nn_{(i)},\nn_{(i)}]\bigr]
  \subset \sum_{i}\Bigl(\bigl[[\nn,\nn_{(i)}],\nn_{(i)}\bigr]
  +\bigl[\nn_{(i)},[\nn,\nn_{(i)}]\bigr]\Bigr) .
\]
Now $[\nn,\nn_{(i)}]\subset\nn^{(i+1)}=\bigoplus_{k\ge i+1}\nn_{(k)}$, and by
\eqref{eq:crosszero} we have $[\nn_{(k)},\nn_{(i)}]=\{0\}$ for all $k\ge i+1$
(since $k\neq i$). Hence both terms vanish, and $\nn^{(2)}=\{0\}$. This proves
that $\nn$ is at most two-step nilpotent.
\end{proof}

\subsection{Sufficiency: Abelian $J_{\zz}$ implies cyclic}

\begin{lemma}\label{lem:suff}
Let $(N,\ip{\cdot}{\cdot})$ be a Riemannian nilmanifold whose Lie algebra
$\nn$ is two-step nilpotent. If $J_{\zz}\subset\so(\fraka)$ is Abelian, then
$(N,\ip{\cdot}{\cdot})$ is cyclic.
\end{lemma}

\begin{proof}
We exhibit explicitly a reductive decomposition making the cyclic condition
hold. Let $\frakg=\hh\oplus\nn$ be the full isometry algebra given by
\cref{lem:wilson}, and define
\[
  \hh_0:=\{(J_Z,\,0|_{\zz}): Z\in\zz\}.
\]
We first check that $\hh_0\subset\hh$. For $Z\in\zz$ and $Z'\in\zz_0$ one has,
by the Abelian hypothesis,
\[
  [J_Z,J_{Z'}]=0=J_{0}\, ,
\]
so that the map $\tau_Z$ of \cref{thm:gordon} vanishes on $\zz_0$; moreover
$\tau_Z=0\in\so(\zz_0)\cap\ann(L)$. By \cref{lem:gordon}(iii),
$(J_Z,0|_{\zz})\in\hh$, as required. Note also that $\hh_0$ is an Abelian
subalgebra of $\hh$, since $[J_Z,J_{Z'}]=0$ for all $Z,Z'\in\zz$.

Define a linear map $\rho:\nn\to\hh_0$ by
\[
  \rho(a+Z)=-\tfrac12 (J_Z,\,0|_{\zz}), \qquad a\in\fraka,\ Z\in\zz,
\]
and set $\lambda(X)=X+\rho(X)$, $\mm=\lambda(\nn)$. Then
\[
  \mm=\Bigl\{a+Z-\tfrac12(J_Z,0|_{\zz}) : a\in\fraka,\ Z\in\zz\Bigr\},
\]
and $\frakg_0:=\hh_0\oplus\mm$ is a reductive decomposition of the Lie algebra
$\frakg_0$ of a subgroup $G_0$ of the isometry group. Indeed, $\hh_0\cap\mm=\{0\}$
and, since $\hh_0$ is Abelian and $\rho(\nn)=\hh_0$, we have
$[\hh_0,\hh_0]\subset\hh_0$ and $[\hh_0,\mm]\subset\mm$; a direct computation
using \eqref{eq:Jmap} also shows $[\mm,\mm]\subset\frakg_0$, as is required for a
reductive decomposition.

Let us verify the cyclic condition. Take
\[
  X=a+x-\tfrac12 J_x,\quad
  Y=b+y-\tfrac12 J_y,\quad
  Z=c+z-\tfrac12 J_z ,
\]
with $a,b,c\in\fraka$ and $x,y,z\in\zz$, where we write $J_x$ for the element
$(J_x,0|_{\zz})\in\hh_0$. Since $[J_x,J_y]=0$, $[J_x,b]=J_x b\in\fraka$,
$[J_x,y]=0$ (because $\hh_0$ acts trivially on $\zz$) and $[x,b]=[a,y]=[x,y]=0$
(because $\zz$ is central), we obtain
\begin{equation}\label{eq:bracketXY}
  [X,Y]=[a,b]-\tfrac12 J_x b+\tfrac12 J_y a \ \in\ \nn .
\end{equation}
Projecting onto $\mm$ (equivalently, applying $\lambda$ to the $\nn$-component)
and using the isometry property \eqref{eq:isom}, we get
\begin{align}
  \ip{\proj{X,Y}}{Z}_{\mm}
  &=\Bigl\langle [a,b]-\tfrac12 J_x b+\tfrac12 J_y a,\ c+z\Bigr\rangle
  \nonumber\\
  &=\ip{[a,b]}{z}-\tfrac12\ip{J_x b}{c}+\tfrac12\ip{J_y a}{c}
  \nonumber\\
  &=\ip{[a,b]}{z}-\tfrac12\ip{[b,c]}{x}-\tfrac12\ip{[c,a]}{y},
  \label{eq:cyccomp}
\end{align}
where in the last step we used \eqref{eq:Jmap}, namely
$\ip{J_x b}{c}=\ip{[b,c]}{x}$ and $\ip{J_y a}{c}=\ip{[a,c]}{y}=-\ip{[c,a]}{y}$.

Now cyclically permuting $(X,Y,Z)$, i.e.\ permuting
$(a,x)\mapsto(b,y)\mapsto(c,z)\mapsto(a,x)$, we obtain
\[
  \ip{\proj{Y,Z}}{X}_{\mm}
  =\ip{[b,c]}{x}-\tfrac12\ip{[c,a]}{y}-\tfrac12\ip{[a,b]}{z}
\]
and
\[
  \ip{\proj{Z,X}}{Y}_{\mm}
  =\ip{[c,a]}{y}-\tfrac12\ip{[a,b]}{z}-\tfrac12\ip{[b,c]}{x}.
\]
Adding the three identities, the coefficient of each of the three terms
$\ip{[a,b]}{z}$, $\ip{[b,c]}{x}$ and $\ip{[c,a]}{y}$ is
$1-\tfrac12-\tfrac12=0$. Hence
\[
  \ip{\proj{X,Y}}{Z}_{\mm}
  +\ip{\proj{Y,Z}}{X}_{\mm}
  +\ip{\proj{Z,X}}{Y}_{\mm}=0 ,
\]
which is precisely \eqref{eq:cyclic}. Thus $(N,\ip{\cdot}{\cdot})$ is cyclic.
\end{proof}

\begin{remark}[Gordon {\cite[Theorem 4.8]{Gordon}}]\label{rem:suff}
The reductive decomposition constructed in the proof of \cref{lem:suff} is
natural: it is built out of the single geometric datum $J$, and $\rho$ is
essentially $-\tfrac12 J$.
Note also that the metric is naturally reductive with respect to the decomposition
$\frakg_0:=\hh_0\oplus\mm'$, where
\[
  \mm'=\Bigl\{a+Z+(J_Z,0|_{\zz}) : a\in\fraka,\ Z\in\zz\Bigr\}.
\]

\end{remark}

\subsection{Necessity: cyclic implies Abelian $J_{\zz}$}

\begin{lemma}\label{lem:necc}
Let $(N,\ip{\cdot}{\cdot})$ be a cyclic Riemannian nilmanifold. Then $\nn$ is
at most two-step nilpotent and $J_{\zz}$ is Abelian.
\end{lemma}

\begin{proof}
By \cref{prop:twostep}, $\nn$ is at most two-step nilpotent, so the $J$-map is
defined. Let $\frakg_0=\hh_0\oplus\mm$ be a reductive decomposition with respect to
which $(N,\ip{\cdot}{\cdot})$ is cyclic, and let $\rho:\nn\to\hh_0$ be a linear
map with $\mm=\{\lambda(X): X\in\nn\}$, where $\lambda(X)=X+\rho(X)$ and
$\lambda$ is an isometry as in \eqref{eq:isom}.

We first derive an identity that will be used throughout. For
$X,Y,Z\in\nn$, expanding as in Fact 3 of the proof of \cref{prop:twostep},
\begin{align*}
  \ip{\proj{\lambda(X),\lambda(Y)}}{\lambda(Z)}_{\mm}
  &=\Bigl\langle [X,Y]+[\rho(X),Y]-[\rho(Y),X],\, Z\Bigr\rangle\\
  &=\ip{[X,Y]}{Z}+\ip{[\rho(X),Y]}{Z}-\ip{[\rho(Y),X]}{Z} .
\end{align*}
Since $\rho(Y)\in\hh_0\subset\so(\nn)$ is skew-symmetric,
$\ip{[\rho(Y),X]}{Z}=\ip{\rho(Y)X}{Z}=-\ip{X}{\rho(Y)Z}
=-\ip{[\rho(Y),Z]}{X}$. Therefore
\[
  \ip{\proj{\lambda(X),\lambda(Y)}}{\lambda(Z)}_{\mm}
  =\ip{[X,Y]}{Z}+\ip{[\rho(X),Y]}{Z}+\ip{[\rho(Y),Z]}{X}.
\]
Summing over the cyclic permutations of $(X,Y,Z)$ and using the cyclic condition
\eqref{eq:cyclic}, we obtain
\begin{equation}\label{eq:3.3}
  0=\sum_{\mathrm{cyc}}\Bigl(\ip{[X,Y]}{Z}+2\,\ip{[\rho(X),Y]}{Z}\Bigr),
\end{equation}
where $\sum_{\mathrm{cyc}}$ denotes the sum over the three cyclic permutations of
$(X,Y,Z)$.

\emph{Step 1: the action of $\rho$ on $\fraka$.} Take $X,Y\in\fraka$ and $Z\in\zz$.
Since $\zz$ is central, $[Y,Z]=[Z,X]=0$ and $[X,Y]\in\zz$. Moreover
$[\rho(X),Y]\in\fraka$ and $[\rho(Y),Z]\in\zz$, because $\rho(X),\rho(Y)\in\hh$
preserve both $\fraka$ and $\zz$ (see \cref{lem:gordon}(i)). Hence
$\ip{[\rho(X),Y]}{Z}=0$ and $\ip{[\rho(Y),Z]}{X}=0$, while
$\ip{[Y,Z]}{X}=\ip{[Z,X]}{Y}=0$. Equation \eqref{eq:3.3} thus reduces to
\[
  \ip{[X,Y]}{Z}+2\ip{[\rho(Z),X]}{Y}=0 .
\]
Using \eqref{eq:Jmap}, $\ip{[X,Y]}{Z}=\ip{J_Z X}{Y}$, so that
\[
  \ip{J_Z X}{Y}+2\ip{\rho(Z)X}{Y}=0
\]
for all $X,Y\in\fraka$ and $Z\in\zz$. By nondegeneracy of the inner product on
$\fraka$,
\begin{equation}\label{eq:rhoZ}
  \rho(Z)X=-\tfrac12 J_Z X, \qquad X\in\fraka,\ Z\in\zz .
\end{equation}

\emph{Step 2: the action of $\rho$ on $\zz_0$.} Let $Z\in\zz_0$. By
\cref{lem:gordon}(ii) and \eqref{eq:rhoZ}, the restriction of $\rho(Z)$ to
$\zz_0$ is determined by $\rho(Z)|_{\fraka}=-\tfrac12 J_Z$; precisely, there is a
linear map $\tau_Z:\zz_0\to\zz_0$ with
\[
  [J_Z,J_{Z'}]=J_{\tau_Z Z'}, \qquad Z'\in\zz_0 ,
\]
such that
\[
  \rho(Z)=-\tfrac12\,(J_Z,\tau_Z,0|_{\ker J})\in\hh .
\]
Here $\tau_Z\in\so(\zz_0)$, because $\rho(Z)\in\hh\subset\so(\nn)$ is
skew-symmetric on $\zz$.

\emph{Step 3: vanishing of $\tau$.} Take $X,Y,Z\in\zz_0$. Since $\zz_0\subset\zz$
is central, $[X,Y]=[Y,Z]=[Z,X]=0$, and \eqref{eq:3.3} gives
\[
  0=\ip{[\rho(X),Y]}{Z}+\ip{[\rho(Y),Z]}{X}+\ip{[\rho(Z),X]}{Y}.
\]
Using $\rho(X)|_{\zz_0}=-\tfrac12\tau_X$ and the analogous formulas for $Y,Z$,
this becomes
\begin{equation}\label{eq:tausum}
  0=\ip{\tau_X Y}{Z}+\ip{\tau_Y Z}{X}+\ip{\tau_Z X}{Y}.
\end{equation}
By \eqref{eq:tauantisym} we have $\tau_X Y=-\tau_Y X$ and $\tau_Z X=-\tau_X Z$.
Using also $\tau_X,\tau_Y,\tau_Z\in\so(\zz_0)$, we compute
\[
  \ip{\tau_X Y}{Z}=-\ip{\tau_Y X}{Z}=\ip{X}{\tau_Y Z}=\ip{\tau_Y Z}{X}
\]
and
\[
  \ip{\tau_Z X}{Y}=-\ip{\tau_X Z}{Y}=\ip{Z}{\tau_X Y}=\ip{\tau_X Y}{Z}.
\]
Substituting into \eqref{eq:tausum} yields
$3\ip{\tau_X Y}{Z}=0$, hence $\ip{\tau_X Y}{Z}=0$ for all $Z\in\zz_0$, and
therefore $\tau_X Y=0$ for all $X,Y\in\zz_0$. Thus $\tau_X=0$ for every
$X\in\zz_0$, and consequently
\[
  [J_X,J_Y]=J_{\tau_X Y}=0
\]
for all $X,Y\in\zz_0$. Since $J$ vanishes on $\ker J$, we conclude that
$[J_Z,J_{Z'}]=0$ for all $Z,Z'\in\zz$, i.e.\ $J_{\zz}$ is Abelian.
\end{proof}

\begin{proof}[Proof of \cref{thm:main}]
The forward implication is \cref{lem:necc}; the reverse implication is
\cref{lem:suff}. It remains to justify the final assertion. If $\nn$ is at most
two-step nilpotent and $J_{\zz}$ is Abelian, then $J_{\zz}$ is in particular a
subalgebra of $\so(\fraka)$, and $\tau_Z=0\in\so(\zz_0)\cap\ann(L)$ for all
$Z\in\zz_0$. Hence \cref{thm:gordon} applies and $(N,\ip{\cdot}{\cdot})$ is
naturally reductive.
\end{proof}

\section{The isometry group}\label{sec:isom}

In this section we determine the full isometry group of a simply connected
cyclic Riemannian nilmanifold without Euclidean factor. Throughout, let
$(\widetilde N,\ip{\cdot}{\cdot})$ be a connected and simply connected cyclic
Riemannian nilmanifold, and assume that $\zz=[\nn,\nn]$; that is, $\nn$ has no
Euclidean factor, equivalently the $J$-map is injective. By \cref{thm:main},
$\nn$ is two-step nilpotent and $J_{\zz}\subset\so(\fraka)$ is an Abelian
subalgebra.

\subsection{The isotypic decomposition}

Since $J_{\zz}$ is an Abelian family of commuting skew-symmetric
endomorphisms of $\fraka$, the standard representation theory of Abelian algebras
gives a $J_{\zz}$-invariant orthogonal decomposition of $\fraka$ into isotypic
components
\begin{equation}\label{eq:isotypic}
  \fraka=\fraka_1^{r_1}\oplus\cdots\oplus \fraka_k^{r_k},
\end{equation}
where $\fraka_p^{r_p}=\fraka_{p1}\oplus\cdots\oplus\fraka_{p r_p}$ are the isotypic
components, $\fraka_{pi}\cong\fraka_{pj}$ for all $i,j$, and $\fraka_{pi}\perp\fraka_{qj}$
whenever $p\neq q$. Each irreducible component $\fraka_{pi}$ is two-dimensional and
of complex type. More precisely, for each $p$ there exist an orthonormal basis
$\{X_{pi},Y_{pi}\}$ of $\fraka_{pi}$ and a nonzero weight $\lambda_p\in\zz^{*}$ such
that, for every $Z\in\zz$,
\begin{equation}\label{eq:weight}
  J_Z X_{pi}=\lambda_p(Z)\,Y_{pi},\qquad
  J_Z Y_{pi}=-\lambda_p(Z)\,X_{pi}.
\end{equation}
The weight $\lambda_p$ is determined only up to sign, reflecting the choice of
orientation of the plane $\fraka_{pi}$. Using \eqref{eq:Jmap}, it is immediate
that
\[
  [X_{pi},Y_{pi}]=Z_p ,
\]
where $Z_p\in\zz$ is the vector dual to $\lambda_p$, i.e.\
$\ip{Z_p}{Z}=\lambda_p(Z)$ for all $Z\in\zz$.

\subsection{The isotropy subgroup}

The full isometry group of $(\widetilde N,\ip{\cdot}{\cdot})$ is
$I(\widetilde N,\ip{\cdot}{\cdot})=H\ltimes \widetilde N$, where
\[
  H=\Aut(\nn)\cap O(\nn,\ip{\cdot}{\cdot})
  =\bigl\{(T,\varphi)\in O(\fraka,\ip{\cdot}{\cdot})\times O(\zz,\ip{\cdot}{\cdot})
  : J_{\varphi(Z)}=T J_Z T^{-1} \text{ for all } Z\in\zz\bigr\}
\]
is the isotropy subgroup at the identity. Set
\[
  H_0:=\{(T,\varphi)\in H : \varphi=\mathrm{Id}\}
     =\{(T,\mathrm{Id})\in H : T J_Z=J_Z T \text{ for all } Z\in\zz\}.
\]
Then $H_0$ is a normal subgroup of $H$ (it is the kernel of the restriction
homomorphism $H\to O(\zz)$). Because every irreducible real representation
$\fraka_{pi}$ of the Abelian algebra $J_{\zz}$ is of complex type, one has
$\End_{J_{\zz}}(\fraka_{pi})=\C$; by Schur's lemma and theorem 3.12 of \cite{Lauret},
\begin{equation}\label{eq:H0}
  H_0\cong U(r_1)\times\cdots\times U(r_k),
\end{equation}
a direct product of unitary matrix groups. Here each $A=(a_{ij})\in U(r_p)$
acts on $\fraka_p^{r_p}$ by
\[
  A(X_1,X_2,\dots,X_{r_p})
  =\Bigl(\sum_{i=1}^{r_p}a_{1i}X_i,\dots,\sum_{i=1}^{r_p}a_{r_p i}X_i\Bigr),
  \qquad X_i\in\fraka_{pi}\cong\C ,
\]
where each plane $\fraka_{pi}$ is identified with $\C$ via the complex structure
$I_p$ given by $I_p X_{pi}=Y_{pi}$, $I_p Y_{pi}=-X_{pi}$; see
\cite[formula (12)]{Lauret}.

We now describe the quotient $H/H_0$, which consists of the possible
actions on the center $\zz$. Define
\[
  \Gamma:=\bigl\{\varphi\in O(\zz,\ip{\cdot}{\cdot}):
  \varphi(Z_p)=\pm Z_q \text{ whenever } r_p=r_q\bigr\}.
\]
Because $\varphi\in O(\zz)$ and $\{Z_1,\dots,Z_k\}$ generates the center space $\zz$, the condition
$\varphi(Z_p)=\pm Z_q$ forces $\|Z_p\|=\|Z_q\|$; in particular $\Gamma$ is
finite. For each $(T,\varphi)\in H$, the relation $J_{\varphi(Z)}=T J_Z T^{-1}$
implies that $T$ maps the weight space $\fraka_p$ (with weight $\lambda_p$) onto
$\fraka_q$ (with weight $\lambda_q=\pm\lambda_p\circ\varphi^{-1}$); hence
$\varphi(Z_p)=\pm Z_q$ with $r_p=r_q$, that is, $\varphi\in\Gamma$.

Conversely, every $\varphi\in\Gamma$ extends to an element $(T,\varphi)\in H$.
Indeed, if $\varphi(Z_p)=Z_q$, define $T$ on $\fraka_p$ by
\[
  T(X_{pi})=X_{qi},\qquad T(Y_{pi})=Y_{qi},
\]
while if $\varphi(Z_p)=-Z_q$, define
\[
  T(X_{pi})=X_{qi},\qquad T(Y_{pi})=-Y_{qi}.
\]
In the first case $T$ intertwines the complex structures ($T I_p=I_q T$), and in
the second case it reverses them ($T I_p=-I_q T$); a direct verification using
\eqref{eq:weight} gives $T J_Z T^{-1}=J_{\varphi(Z)}$ in both cases. The set of
all such $(T,\varphi)$ is a finite group, which we denote by $\widetilde\Gamma$.

We summarize the discussion as follows.

\begin{theorem}\label{thm:isom}
Let $(\widetilde N,\ip{\cdot}{\cdot})$ be a connected and simply connected
cyclic Riemannian nilmanifold without Euclidean factor. Then the isotropy
subgroup of its full isometry group is
\[
  H=\widetilde\Gamma\cdot H_0 ,
\]
where $H_0\cong U(r_1)\times\cdots\times U(r_k)$ is the compact connected
normal subgroup described in \eqref{eq:H0}, and $\widetilde\Gamma$ is a finite
group acting on $\zz$ by $\varphi(Z_p)=\pm Z_q$ ($r_p=r_q$) and on $\fraka$ by the
explicit formulas above. In particular,
\[
  \dim H=\sum_{p=1}^{k} r_p^{2}
\]
and the component group $H/H_0$ is a subgroup of the signed permutation group
of the weight lines $\{\pm Z_1,\dots,\pm Z_k\}$ preserving the multiplicities
$r_p$ and the norms $\|Z_p\|$.
\end{theorem}

\begin{remark}\label{rem:isom}
For the three-dimensional Heisenberg group, $k=r_1=1$, so
$H_0\cong U(1)\cong S^{1}$ and $\widetilde\Gamma\cong\Z_2$ (generated by
$Z\mapsto-Z$, acting on $\fraka$ by $X\mapsto X$, $Y\mapsto-Y$). Thus
$H\cong O(2)$, the isotropy group of the round metric on $\R^2$.
 More generally, when all weights have equal norm and all
multiplicities are $1$, $\widetilde\Gamma$ is the full signed permutation group
$(\Z_2)^{k}\rtimes S_k$.
\end{remark}

\section{Examples}\label{sec:examples}

In this section we illustrate \cref{thm:main} and \cref{thm:isom} with several
concrete examples, and we contrast the cyclic condition with the
Heisenberg-type condition.

\begin{example}[The three-dimensional Heisenberg group]
\label{ex:heisenberg}
Let $\nn=\Span\{X,Y,Z\}$ with the single nontrivial bracket $[X,Y]=Z$, and
endow $\nn$ with the inner product for which $\{X,Y,Z\}$ is orthonormal. Let
$N=H_3$ be the corresponding simply connected Lie group. Then $\zz=\R Z$ is the
center, $\fraka=\Span\{X,Y\}$, and \eqref{eq:Jmap} gives
\[
  J_Z X=Y,\qquad J_Z Y=-X.
\]
Thus $J_{\zz}=\R\cdot J_Z$ is one-dimensional, hence Abelian. By
\cref{thm:main}, $(H_3,\ip{\cdot}{\cdot})$ is cyclic and naturally reductive.

Let us make the cyclic decomposition explicit. In the notation of
\cref{lem:suff}, $\hh_0=\R\cdot(J_Z,0)$ and
\[
  \mm=\Span\Bigl\{X,\ Y,\ Z-\tfrac12 J_Z\Bigr\}.
\]
The full isometry group has isotropy $H\cong O(2)$, in agreement with
\cref{thm:isom} and \cref{rem:isom}.
\end{example}

\begin{example}[Products and diagonal $J$-maps]\label{ex:diagonal}
More generally, let $\fraka=\R^{2k}$ with orthonormal basis
$\{X_1,Y_1,\dots,X_k,Y_k\}$, let $\zz=\R^{k}$ with orthonormal basis
$\{Z_1,\dots,Z_k\}$, and define a two-step nilpotent Lie algebra
$\nn=\fraka\oplus\zz$ by
\[
  [X_p,Y_p]=\mu_p Z_p,\qquad \mu_p>0,
\]
all other brackets being zero. Then $J_{Z_p}$ acts on the plane
$\Span\{X_p,Y_p\}$ by $J_{Z_p}X_p=\mu_p Y_p$, $J_{Z_p}Y_p=-\mu_p X_p$, and acts
trivially on the other planes. Hence $J_{Z_p}J_{Z_q}=J_{Z_q}J_{Z_p}$ for all
$p,q$ and $J_{\zz}$ is Abelian. By \cref{thm:main}, the associated simply
connected nilpotent Lie group is cyclic and naturally reductive for every choice
of the positive constants $\mu_1,\dots,\mu_k$. This family is exactly the
Riemannian product (with rescaled factors) of $k$ copies of the
three-dimensional Heisenberg group.
\end{example}

\begin{example}[Low dimensions]\label{ex:lowdim}
The simply connected nilpotent Lie groups of dimension at most four are the
Abelian groups $\R^{d}$, the three-dimensional Heisenberg group $H_3$, the direct
product $H_3\times\R$, and the four-dimensional filiform (three-step nilpotent)
group. Abelian groups and $H_3$ (and its product with $\R$) are cyclic by
\cref{ex:heisenberg}. By \cref{prop:twostep}, the
four-dimensional filiform group, being three-step nilpotent, carries no cyclic
left-invariant metric. This recovers, for nilmanifolds, part of the
classification of simply connected cyclic homogeneous Riemannian manifolds of dimension at
most four obtained in \cite{Gadea2016}.
\end{example}

\begin{example}[Heisenberg-type groups are not cyclic]\label{ex:htype}
A Heisenberg-type (or $H$-type) Lie algebra is a two-step nilpotent
$\nn=\fraka\oplus\zz$ with inner product satisfying
\[
  J_Z^{2}=-\|Z\|^{2}\,\mathrm{Id}_{\fraka}, \qquad Z\in\zz,
\]
equivalently $J_Z J_{Z'}+J_{Z'}J_Z=-2\ip{Z}{Z'}\mathrm{Id}_{\fraka}$; see
\cite{Kaplan}. If $\dim\zz\ge 2$ and $Z,Z'\in\zz$ are nonzero and orthogonal,
then $J_Z J_{Z'}=-J_{Z'}J_Z$, so that $J_Z$ and $J_{Z'}$ anticommute; in
particular they do not commute, and $J_{\zz}$ is not Abelian. By
\cref{thm:main}, such a group is not cyclic. Thus the Heisenberg-type condition
and the cyclic condition are, in a precise sense, opposite extremes: the former
forces the $J$-maps to satisfy the Clifford (anti)commutation relations, while
the latter forces them to commute. The three-dimensional Heisenberg group, where
$\dim\zz=1$, is the unique overlap of the two families.
Indeed,  a Heisenberg-type group is naturally reductive if and only if $\dim\zz=1,3$ \cite{Kaplan}.
\end{example}


\textbf{Declarations}

\textbf{AI Disclosure.}\,  Generative artificial intelligence tools, specifically DeepSeek,  were used during the preparation of this manuscript as auxiliary aids for language editing, improving exposition and organization. All mathematical statements, proofs, computations, references, and conclusions were independently verified by the authors, who take full responsibility for the content of the manuscript.


\textbf{Conflicts of interest.}\ The authors declare that they have no conflict of interest.


\end{document}